\documentclass[10pt, a4paper, reqno, oneside]{amsart}

\usepackage{amsmath, amsfonts, amsthm, amssymb, mathtools, enumerate, multicol, scalefnt, relsize}
\usepackage[mathscr]{euscript}
\usepackage{mathbbol}
\usepackage[latin1]{inputenc}
\usepackage{graphicx}
\usepackage[all]{xy}
\usepackage{enumitem}
\usepackage{autobreak}
\setlist[itemize]{noitemsep, topsep=1pt, leftmargin=20pt}
\usepackage{xfrac}
\usepackage{todonotes}

\usepackage{fullpage}

\usepackage[hypertexnames=false,
backref=page,
    pdfpagemode=UseNone,
    breaklinks=true,
    extension=pdf,
    colorlinks=true,
    linkcolor=blue,
    citecolor=blue,
    urlcolor=blue,
]{hyperref}

\newcommand\bcdot{\ensuremath{
  \mathchoice
   {\mskip\thinmuskip\lower0.2ex\hbox{\scalebox{1.6}{$\cdot$}}\mskip\thinmuskip}}
   {\mskip\thinmuskip\lower0.2ex\hbox{\scalebox{1.6}{$\cdot$}}\mskip\thinmuskip}
   {\lower0.3ex\hbox{\scalebox{1.2}{$\cdot$}}}
   {\lower0.3ex\hbox{\scalebox{1.2}{$\cdot$}}}
}
\theoremstyle{plain}
\newtheorem{theo}{Theorem}[section]

\newtheorem{prop}[theo]{Proposition}

\theoremstyle{definition}
\newtheorem{rem}[theo]{Remark}

\newtheorem{definition}[theo]{Definition}

\theoremstyle{plain}
\newtheorem{lemma}[theo]{Lemma}
\newtheorem{theorem}[theo]{Theorem}
\newtheorem{corollary}[theo]{Corollary}
\newtheorem{proposition}[theo]{Proposition}

\theoremstyle{definition}

\theoremstyle{plain}
\newtheorem{thmint}{Theorem}

\newtheorem{corint}[thmint]{Corollary}

\renewcommand{\=}{\coloneqq}

\newcommand{\td}{\mathtt{d}}

\renewcommand{\a}{\alpha}
\renewcommand{\b}{\beta}

\renewcommand{\d}{\delta}
\newcommand{\e}{\varepsilon}

\newcommand{\g}{\gamma}

\renewcommand{\l}{\lambda}

\newcommand{\D}{\Delta}

\renewcommand{\S}{\Sigma}

\DeclareSymbolFontAlphabet{\mathbb}{AMSb}
\DeclareSymbolFontAlphabet{\mathbbl}{bbold}

\newcommand{\bR}{\mathbb{R}}

\newcommand{\bN}{\mathbb{N}}

\newcommand{\fG}{\mathsf{G}}
\newcommand{\fH}{\mathsf{H}}

\newcommand{\fK}{\mathsf{K}}

\newcommand{\fN}{\mathsf{N}}

\newcommand{\fS}{\mathsf{S}}
\newcommand{\fT}{\mathsf{T}}

\newcommand{\fSO}{\mathsf{SO}}

\newcommand{\fSU}{\mathsf{SU}}

\renewcommand{\gg}{\mathfrak{g}}
\newcommand{\gh}{\mathfrak{h}}

\newcommand{\gk}{\mathfrak{k}}

\newcommand{\gm}{\mathfrak{m}}
\newcommand{\gn}{\mathfrak{n}}

\newcommand{\gt}{\mathfrak{t}}

\newcommand{\so}{\mathfrak{so}}
\newcommand{\su}{\mathfrak{su}}

\newcommand\Sym{\mathrm{Sym}}

\newcommand{\cC}{\mathcal{C}}

\newcommand{\cR}{\mathcal{R}}

\newcommand{\eB}{\EuScript{B}}

\newcommand{\eM}{\EuScript{M}}

\newcommand{\eW}{\EuScript{W}}

\DeclareMathOperator\Tr{Tr}
\DeclareMathOperator\inj{inj}
\DeclareMathOperator\Lie{Lie}
\DeclareMathOperator\End{End}
\DeclareMathOperator\Hess{Hess}
\DeclareMathOperator\scal{scal} 
\DeclareMathOperator\Ad{Ad}
\DeclareMathOperator\ad{ad}
\DeclareMathOperator\Id{Id}

\DeclareMathOperator\diff{d\!}

\newcommand{\wt}{\widetilde}

\newcommand{\Ric}{\operatorname{Ric}}
\newcommand{\Rm}{\operatorname{Rm}}

\newcommand{\zero}{\operatorname{o}}

\makeatletter
\@namedef{subjclassname@2020}{\textup{2020} Mathematics Subject Classification}
\makeatother

\allowdisplaybreaks[3]

\title[]{Collapsed ancient solutions of the Ricci flow \\ on compact homogeneous spaces}

\author{Francesco Pediconi}
\address[Francesco Pediconi]{Dipartimento di Matematica e Informatica ``Ulisse Dini'' \\ Universit\`a degli Studi di Firenze, viale Morgagni 67/a, 50134 Firenze, Italy}
\curraddr{Department of Mathematics, Aarhus University, Ny Munkegade 118, 8000 Aarhus C, Denmark}
\email{francesco.pediconi@unifi.it}
\email{francesco.pediconi@math.au.dk}

\author{Sammy Sbiti}
\address[Sammy Sbiti]{Department of Mathematics, University of Pennsylvania, 209 South 33rd St, Philadelphia, PA, 19104-6395, USA}
\email{sbiti@sas.upenn.edu }

\thanks{The first-named author is member of GNSAGA of INdAM and has been supported by the project PRIN 2017 ``Real and Complex Manifolds: Topology, Geometry and holomorphic dynamics'' (code 2017JZ2SW5).}

\begin{document}

\begin{abstract}
We prove a general existence theorem for collapsed ancient solutions to the Ricci flow on compact homogeneous spaces and we show that they converge in the Gromov-Hausdorff topology, under a suitable rescaling, to an Einstein metric on the base of a torus fibration. This construction generalizes all previous known examples in the literature.
\end{abstract}

\maketitle

\section{Introduction}

Given a {connected} smooth manifold $M$, a solution to Hamilton's {\it Ricci flow} is a smooth path of Riemannian metrics $g(t)$ satisfying the geometric non-linear PDE
$$
\tfrac{\partial}{\partial t} g(t) = -2\Ric_{\mathsmaller M}(g(t)) \,\, .
$$
Up to a time-dependent family of diffeomorphisms, it is equivalent to a parabolic PDE and so, similar to the heat equation, the Ricci flow has regularizing properties for Riemannian metrics making it a fundamental tool in the study of classification-type problems in geometry and topology.
Solutions that are defined on a time interval $(-\infty, 0]$ have a special significance and are called {\it ancient solutions}.
Since the Ricci flow is equivalent to a parabolic PDE, the backwards equation is generally ill-posed and hence ancient solutions are rare. \smallskip

Now let $\fG$ be a compact Lie group acting transitively and almost effectively on $M$, so that $M$ is compact as well. Then, by diffeomorphism invariance, the Ricci flow induces a dynamical system on the space of $\fG$-invariant Riemannian metrics $\eM^{\fG}_{\mathsmaller M}$ on $M$. We recall that an ancient solution to the $\fG$-homogeneous Ricci flow $g(t)$ is said to be {\it non-collapsed} if the curvature-normalized metrics $|\Rm_{\mathsmaller M}(g(t))|_{g(t)}g(t)$ have a uniform lower injectivity radius bound, otherwise it is said to be {\it collapsed}. In \cite{BLS}, the authors proved that non-collapsed ancient solutions on a compact homogeneous space must emanate from a homogeneous Einstein metric on the same space. Moreover, the {\it normalized Ricci flow}
$$
\tfrac{\partial}{\partial t} \tilde{g}(t) = -2\Ric_{\mathsmaller M}^0(\tilde{g}(t)) \,\, 
$$
on the subspace of $\fG$-invariant, unit volume, Riemannian metrics $\eM^{\fG}_{{\mathsmaller M},1} \subset \eM^{\fG}_{\mathsmaller M}$, which is equivalent to the Ricci flow up to rescaling and time reparametrization, is the gradient flow of the scalar curvature functional, whose critical points are $\fG$-invariant, unit volume, Einstein metrics on $M$. Therefore, non-collapsed ancient solutions to the $\fG$-homogeneous Ricci flow always exist whenever $M$ admits a $\fG$-invariant, $\fG$-unstable Einstein metric (see \cite[Lemma 5.4]{BLS}). \smallskip

On the other hand, it is known that a compact homogeneous space admits collapsed ancient solutions to the homogeneous Ricci flow only if it is the total space of a homogeneous torus bundle (see {\it e.g.\ }Proposition \ref{prop:coll-0PS}). Examples of such solutions have been found, {\it e.g.\ }in \cite{S,Bu,BLS,BKN,LW}, on a case by case basis, and up to now there have been few general existence theorems. \smallskip

The main result of our paper is the following.

\begin{thmint} \label{main-A}
Let $\fH \subsetneq \fT\fH\subsetneq\fG$ be compact, connected Lie groups, where $\fT$ is a maximal torus of a compact complement of $\fH$ in the normalizer $\fN_\fG(\fH)$ with $d = \dim(\fT) \geq 1$. For any $\fG$-invariant, unit volume Einstein metric $\bar{g}$ on $N = \fG/\fT\fH$ of coindex $q$, there exists a $\big(\frac{d(d+1)}2+q-1\big)$-parameter family of collapsed ancient solutions to the homogeneous Ricci flow on $M = \fG/\fH$ that, under a suitable rescaling, shrink the fibers of $\fT \to M \to N$ and converge to $(N,\bar{g})$ in the Gromov-Hausdorff topology as $t \to -\infty$.
\end{thmint}

Here, the coindex of the Einstein metric $\bar{g}$ is defined as its coindex as a critical point of the scalar curvature functional on the space $\eM^{\fG}_{{\mathsmaller N},1}$ of unit volume $\fG$-invariant metrics on $N$ (see Section \ref{subsec:nondeg}). We mention that all the ancient solutions obtained by means of Theorem \ref{main-A} are of {\it submersion type} with respect to the homogeneous torus fibration $\fT \to M \to N$, and we expect this property to be true in general (see {\it e.g.\ }\cite{S}). Here, maximality of $\fT$ and Schur's Lemma guarantee that submersion metrics are preserved by the Ricci flow, which is crucial to the proof. Furthermore we stress that, along the solutions obtained by our theorem, the metric on the fibers $\fT$ and the base $N$ will in general change. Notice also that Theorem \ref{main-A} holds true even when $\bar{g}$ has coindex $q=0$, {\it i.e.\ }it is a local maximum of scalar curvature on $\eM^{\fG}_{{\mathsmaller N},1}$ (compare with \cite[Lemma 5.4]{BLS}). We mention that Theorem \ref{main-A} holds without the connectedness assumption on $\fH$ and $\fG$, but in such case the dimension of the space of ancient solutions on $M$ depends on the adjoint representation of $\fH$ on the Lie algebra $\gt \= \Lie(\fT)$ (see Section \ref{subsect:existence}). Lastly, we do not know, expect in some special cases, whether the solutions found by means of Theorem \ref{main-A} are isometric or not. \smallskip

We illustrate our result with the following series of examples:

\begin{corint} \label{main-B} \hfill \par
\begin{itemize}
\item[a)] On $\fSU(3)$, $\fSU(3)/\fS^1$, $\fSU(4)$, $\fSU(4)/\fS^1$, $\fSU(4)/\fT^2$, $\fG_2$ and $\fG_2/\fS^1$ there exists a 3, respectively, 1, 7, 4, 2, 3 and 1-parameter family of ancient solutions to the Ricci flow collapsing, under a suitable rescaling, to a K\"{a}hler-Einstein metric on a full flag manifold.
\item[b)] On $\fSU(n)/\fT^{n-1-k}$, with $n \geq 3$ and $1\leq k\leq n-1$, there exists a $\big(\frac{k(k+1)}2+n-2\big)$-parameter family of ancient solutions to the Ricci flow collapsing to the normal Einstein metric on $\fSU(n)/\fT^{n-1}$.
\item[c)] On $\fSO(4)$ and $\fSO(4)/\fS^1$ there exists a 3, respectively, 1-parameter family of ancient solutions to the Ricci flow collapsing to the normal Einstein metric on $\fSO(4)/\fT^2$. 
\end{itemize}
\end{corint}

Notice that the the circles $\fS^1\subset\fSO(4)$ in Corollary \ref{main-B} can be chosen with arbitrary slope and that the manifolds $\fSO(4)/\fS^1$ are diffeomorphic to $S^3\times S^2$ (see {\it e.g.\ }\cite{WZ}). In particular we obtain infinitely many families of ancient solutions on $S^3\times S^2$ that are homogeneous under inequivalent group actions.
Moreover, let us also observe that in \cite{Boe1} it was shown that under the assumptions of Theorem \ref{main-A} most homogeneous spaces $\fG/\fT\fH$ admit homogeneous Einstein metrics with large coindex plus nullity (see also \cite{WaZ,BWZ,BK}), to which Theorem \ref{main-A} can be applied. \smallskip

We remark that Theorem \ref{main-A} allows us to reconstruct all known examples of collapsed homogeneous ancient solutions to the Ricci flow, such as those in \cite{S,Bu,BLS,BKN,LW}.
In \cite{BKN} and \cite{LW}, the authors construct ancient solutions which consist of submersion metrics $F\to M\to B$ where one assumes either that $F,M,B$ admit Einstein metrics, $F$ is a torus, or $B$ is a product of K\"{a}hler-Einstein metrics. In these constructions the metric on the base remains fixed, or in the latter case stays within the set of products of K\"{a}hler-Einstein metrics. On the contrary, along the ancient solutions provided by Theorem \ref{main-A}, the induced metric on the base $N$ varies and does not necessarily remain Einstein, as opposed to the solutions constructed in \cite{BKN} and those constructed in \cite{LW}. Furthermore in our situation $M$ need not admit an invariant Einstein metric and $\bar{g}$ need not be K\"ahler-Einstein.

The strategy of the proof of Theorem \ref{main-A} is as follows. First we observe that, since the torus $\fT$ is assumed to be maximal, the subspace of $\fG$-invariant submersion metrics with respect to $\fT \to M \to N$ is preserved by the Ricci flow, and that it can be analytically extended to the bigger space of {\it generalized submersion metrics}, which contains a neighborhood of the collapsed metric $0 \oplus \bar{g}$. We also compactify the problem by projecting the Ricci flow to the unit sphere of generalized submersion metrics, with respect to some carefully chosen background inner product depending on $\bar{g}$, obtaining what we call the {\it $\bar{g}$-projected Ricci flow}. Then, we use an enhanced version of the Stable Manifold Theorem to produce ancient solutions to the $\bar{g}$-projected Ricci flow which collapse to $0 \oplus \bar{g}$ as $t \to -\infty$. Finally, we prove that some of these solutions of generalized submersion metrics are actually positive definite, and that the corresponding solutions to the Ricci flow, obtained by rescaling and time reparametrization, are still ancient.

\smallskip
The paper is organized as follows.
In Section \ref{sect:prel}, we recall some facts about compact homogeneous spaces, toral $\fH$-subalgebras and ancient solutions to the Ricci flow.
In Section \ref{sect:projRF}, we introduce the fundamental tools for proving our main result, namely the space of generalized submersion metrics and the projected Ricci flow.
In Section \ref{sect:mainproof} we prove Theorem \ref{main-A}.
In Section \ref{sect:examples} we construct examples of collapsed ancient solutions and prove Corollary \ref{main-B}.

\medskip

\noindent {\itshape Acknowledgements.\ }The authors would like to thank Christoph B\"ohm and Wolfgang Ziller for their invaluable comments and suggestions. They also would like to thank the anonymous referee for their careful reading of the manuscript.

\medskip
\section{Preliminaries on compact homogeneous spaces} \label{sect:prel}
\setcounter{equation} 0

\subsection{The space of invariant metrics} \hfill \par

Let $M=\fG/\fH$ be a compact, connected and almost-effective $m$-dimensional homogeneous space, where $\fG$ is a compact Lie group and $\fH$ a closed subgroup. Furthermore assume that $M$ is not a torus. Notice that neither $\fG$ nor $\fH$ are assumed to be connected.

Fix an $\Ad(\fG)$-invariant Euclidean inner product $Q$ on the Lie algebra $\gg \= \Lie(\fG)$ and denote by $\gm$ the $Q$-orthogonal complement of $\gh \= \Lie(\fH)$ in $\gg$. By means of the canonical identification $\gm \simeq T_{e\fH}M$ given by the evaluation map
$$
V \mapsto V^*_{e\fH} \= \tfrac{\diff}{\diff s} \exp(sV)\fH\big|_{s=0} \,\, ,
$$
we identify any $\fG$-invariant tensor field on $M$ with the corresponding $\Ad(\fH)$-invariant tensor on $\gm$. The restriction $Q_{\gm} \= Q|_{\gm \otimes \gm}$ defines a {\it normal} $\fG$-invariant Riemannian metric on $M$.

We denote by $\eM^{\fG}_{\mathsmaller M}$ the set of $\fG$-invariant Riemannian metrics on $M$, which is identified with the linear space of $Q_{\gm}$-symmetric, $\Ad(\fH)$-invariant, positive-definite endomorphisms of $\gm$, {\it i.e.\ }
\begin{equation} \label{spaceinvmetrics}
\eM^{\fG}_{\mathsmaller M} = \Sym_+(\gm,Q_{\gm})^{\Ad(\fH)} \,\, ,
\end{equation}
by means of the correspondence
\begin{equation} \label{endmetric}
g \mapsto P_g \,\, , \quad Q_{\gm}(P_g.V_1,V_2) \=  g(V_1,V_2) \quad {\text{ for any } V_1, V_2 \in \gm} \,\, .
\end{equation}
From now on, we will always {identify a metric with the associated endomorphism via \eqref{endmetric}.}

We recall that \eqref{spaceinvmetrics} provides the set $\eM^{\fG}_{\mathsmaller M}$ with a structure of finite-dimensional smooth manifold. Moreover, the natural $L^2$-metric defined by
$$
\langle B_1, B_2\rangle_P \= \det(P)^{\frac12}\Tr(P^{-1}.B_1.P^{-1}.B_2) \quad \text{ for any } B_1, B_2 \in T_P\eM^{\fG}_{\mathsmaller M} = \Sym(\gm,Q_{\gm})^{\Ad(\fH)}
$$
turns $\eM^{\fG}_{\mathsmaller M}$ into a Riemannian symmetric space of non-compact type and the subset
$$
\eM^{\fG}_{{\mathsmaller M},1} \= \{P \in \Sym_+(\gm,Q_{\gm})^{\Ad(\fH)} : \det(P)=1 \}
$$
of unit volume $\fG$-invariant Riemannian metrics into a totally geodesic submanifold.  \smallskip

For any Riemannian metric $P \in \eM^{\fG}_{\mathsmaller M}$, we consider the $\Ad(\fH)$-invariant map
$$
S_{\mathsmaller M}(P) : \gm \to \End(\gm)
$$
defined by (see \cite[Thm.\ 3.3, Ch.\ X]{KN2})
\begin{multline} \label{eq:SM}
-2Q_{\gm}(S_{\mathsmaller M}(P)(V_1).V_2,V_3) \= Q_{\gm}([V_1,V_2]_{\gm},V_3)  \\
+Q_{\gm}([P^{-1}.V_3,V_1]_{\gm},P.V_2) +Q_{\gm}([P^{-1}.V_3,V_2]_{\gm},P.V_1) \,\, .
\end{multline}
Here, the symbol $[V_1,V_2]_{\gm}$ denotes the $Q$-orthogonal projection of $[V_1,V_2]$ on $\gm$. The map $S_{\mathsmaller M}(P)$, which is denoted by $-\Lambda_{\gm}$ in \cite{KN2}, corresponds to the $\fG$-invariant $(1,2)$-tensor field on $M$ given by the difference between the Ambrose-Singer connection associated to the reductive decomposition $\gg = \gh + \gm$ and the Levi-Civita connection (see {\it e.g.\ }\cite{Ped3}). It is worth mentioning that this tensor encodes all the geometric information about the metric $P$. Indeed, following \cite[Prop.\ 2.3, Ch.\ X]{KN2}, the {\it Riemannian curvature tensor} $\Rm_{\mathsmaller M}(P)$ of $P$ is explicitly expressed in terms of $S_{\mathsmaller M}(P)$ by
\begin{equation} \label{eq:Rm}
\Rm_{\mathsmaller M}(P)(V_1,V_2) = \ad([V_1,V_2]_{\gh}) -[S_{\mathsmaller M}(P)(V_1),S_{\mathsmaller M}(P)(V_2)] -S_{\mathsmaller M}(P)([V_1,V_2]_{\gm}) \,\, ,
\end{equation}
where again the $[V_1,V_2]_{\gh}$ denotes the $Q$-orthogonal projection of $[V_1,V_2]$ on $\gh$. Consequently, the {\it Ricci curvature} $\Ric_{\mathsmaller M}(P)$ of $P$ is
\begin{equation} \label{eq:Ric}
Q_{\gm}(\Ric_{\mathsmaller M}(P).V_1,V_2) \= \Tr\!\big(\Rm_{\mathsmaller M}(P)(V_1,\,\cdot\,).V_2\,\big)
\end{equation}
and the {\it scalar curvature} $\scal_{\mathsmaller M}(P)$ of $P$
\begin{equation} \label{eq:scal}
\scal_{\mathsmaller M}(P) \= \Tr(P^{-1}.\Ric_{\mathsmaller M}(P)) \,\, .
\end{equation}
Notice that, according to \eqref{eq:Ric}, we denote by $\Ric_{\mathsmaller M}$ the endomorphism obtained by raising an index of the Ricci bilinear form by means of the background metric $Q$. Therefore, the standard ``Ricci endomorphism'' corresponds in our notation to $P^{-1}.\Ric_{\mathsmaller M}(P)$.

We also denote by
\begin{equation} \label{eq:tracelessRic}
\Ric^0_{\mathsmaller M}(P) \=  \Ric_{\mathsmaller M}(P) - \frac{\scal_{\mathsmaller M}(P)}{m}P
\end{equation}
the {\it traceless Ricci curvature of $P$} and we recall that $P$ is said to be {\it Einstein} if $\Ric^0_{\mathsmaller M}(P)=0$.

We finally mention that Einstein metrics are the critical points of the {\it normalized scalar curvature functional}
$$
\wt{\scal}_{\mathsmaller M} : \eM^{\fG}_{\mathsmaller M} \to \bR \,\, , \quad \wt{\scal}_{\mathsmaller M}(P) \= \det(P)^{\frac1m}\scal_{\mathsmaller M}(P) \,\, .
$$
Indeed, following \cite[Ch.\ 4]{Bes} the differential of $\wt{\scal}_{\mathsmaller M}$ at $P \in \eM^{\fG}_{\mathsmaller M}$ in the direction of $B \in T_P\eM^{\fG}_{\mathsmaller M}$ is
\begin{equation} \label{diffscal}
\diff\,\wt{\scal}_{\mathsmaller M}|_P(B) = -\det(P)^{\frac{2-m}{2m}}\langle \Ric^0_{\mathsmaller M}(P), B \rangle_P
\end{equation}
and so $\diff\,\wt{\scal}_{\mathsmaller M}|_P=0$ if and only if $P$ is Einstein.

\subsection{Homogeneous torus bundles and the coindex of Einstein metrics} \label{subsec:nondeg} \hfill \par

Let us consider a {\it toral $\fH$-subalgebra $\gk$ of $\gg$}, that is, an $\Ad(\fH)$-invariant Lie subalgebra of $\gg$ which lies properly between $\gh$ and $\gg$ such that $[\gk,\gk] \subset \gh$. Then, if we denote by $\fK^{\zero}$ the connected Lie subgroup of $\fG$ with Lie algebra equal to $\gk$, it turns out that the subgroup $\fK$ generated by $\fH$ and $\fK^{\zero}$ is a (not necessarily closed) Lie subgroup of $\fG$ and $\fT \= \fK/\fH$ is a (immersed) torus in $M$. This gives rise to a (locally defined) homogeneous torus fibration
\begin{equation} \label{torusfib}
\fT = \fK/\fH \to M =\fG/\fH \to N \= \fG/\fK \,\, .
\end{equation}
For more details on this construction, see {\it e.g.\ }\cite[Sect.\ 4]{Boe1}, \cite[Sect.\ 3]{Ped1} and \cite[Prop.\ 6.1]{Ped2}.

At the Lie algebra level, we get the $Q$-orthogonal decomposition
\begin{equation} \label{eq:dec}
\gg=\underbrace{\gh + \gt}_{\gk} + \!\!\!\!\!\!\!\!\!\overbrace{\phantom{aai}\gn}^{\gm} \,\, , \quad \text{ with } \,\,\, \gt \= \Lie(\fT) \,\, , \,\,\, \gn \simeq T_{e\fK}N \,\, .
\end{equation}
We recall that a metric $P \in \eM^{\fG}_{\mathsmaller M}$ is called {\it $\gk$-submersion metric} if  it preserves the decomposition $\gm = \gt + \gn$ and its restriction to the subspace $\gn$ is $\Ad(\fK)$-invariant. We denote by $\eM^{\fG}_{\mathsmaller M}(\gk)$ the subset of all the $\gk$-submersion metrics and observe that it naturally splits as
\begin{equation} \label{def:invsubmetrics}
\eM^{\fG}_{\mathsmaller M}(\gk) = \Sym_+(\gt,Q_{\gt})^{\Ad(\fH)} \oplus \Sym_+(\gn,Q_{\gn})^{\Ad(\fK)} \,\, , \quad P = P_{\gt} \oplus P_{\gn} 
\end{equation}
where $Q_{\gt} \= Q|_{\gt \otimes \gt}$ and $Q_{\gn} \= Q|_{\gn \otimes \gn}$. Notice that any $P \in \eM^{\fG}_{\mathsmaller M}(\gk)$ turns the (locally) homogeneous torus fibration \eqref{torusfib} into a Riemannian submersion with totally geodesic fibers (see {\it e.g.\ }\cite[Sect.\ 3.2]{Ped1}). {Notice that all the metrics in $\eM^{\fG}_{\mathsmaller M}(\gk)$ are invariant under the action of the larger group $\fG{\times}\fT$, which acts on $M = \fG/\fH$ via $(a,n) \cdot b\fH \= abn^{-1}\fH$ with isotropy at the origin $\fH\D\fT \= \{(hn,n) : h \in \fH, n \in \fT \}$.}

Let us consider now a {\it maximal toral $\fH$-subalgebra} of $\gg$, {\it i.e.\ }a toral $\fH$-subalgebra $\gk$ of $\gg$ such that $\fT = \fK/\fH$ is a maximal torus of a compact complement of $\fH^{\zero}$ in $\fN_\fG(\fH^{\zero})^{\zero}$. Here, we denote by $\fH^{\zero}$ the identity component of $\fH$ and by $\fN_\fG(\fH^{\zero})^{\zero}$ the identity component of the normalizer of $\fH^{\zero}$ in $\fG$. Notice that this condition implies that $\fK$ is closed in $\fG$ and hence $N = \fG/\fK$ is a compact homogeneous space. Moreover, it also implies the following

\begin{lemma} \label{lemma:almtriv}
The complement $\gn$ in \eqref{eq:dec} does not contain any $\Ad(\fK)$-invariant submodule on which $\Ad(\fK^{\zero})$ acts trivially.
\end{lemma}

\begin{proof}
Let $\tilde{\gn} \subset \gn$ be an $\Ad(\fK)$-invariant submodule such that $\Ad(\fK^{\zero}).X=\{X\}$ for any $X \in \tilde{\gn}$. Then, this implies that $\tilde{\gk} \= \gk + \tilde{\gn}$ is a toral $\fH$-subalgebra of $\gg$. Since $\gk$ is assumed to be maximal, it follows that $\tilde{\gk} = \gk$ and so $\tilde{\gn} \subset \gk$. Since $\gk$ and $\gn$ are $Q$-orthogonal, we get $\tilde{\gn} = \{0\}$.
\end{proof}

Let now $\bar{P}_{\gn} \in \eM^{\fG}_{{\mathsmaller N},1}$ be a unit volume Einstein metric on $N$. Then $\Ric^0_{\mathsmaller N}(\bar{P}_{\gn})=0$ and so, by \eqref{diffscal}, it follows that
\begin{equation} \label{diffscal2}
\Hess\!\big(\!\scal_{\mathsmaller N}\!|_{\eM^{\fG}_{{\mathsmaller N},1}}\big)\big|_{\bar{P}_{\gn}}(B_1,B_2) = -\big\langle \!\diff\big(\Ric^0_{\mathsmaller N}|_{\eM^{\fG}_{{\mathsmaller N},1}}\big)\big|_{\bar{P}_{\gn}}(B_1), B_2 \big\rangle_{\bar{P}_{\gn}}
\end{equation}
for any $B_1,B_2 \in T_{\bar{P}_{\gn}}\eM^{\fG}_{{\mathsmaller N},1}$. Therefore, in virtue of \eqref{diffscal2} and \cite[Def.\ 3.14]{Lau}, we recall the following notion of coindex for invariant Einstein metrics on $N$.

\begin{definition} \label{def:coindex}
The {\it coindex of a unit volume Einstein metric $\bar{P}_{\gn} \in \eM^{\fG}_{{\mathsmaller N},1}$} is its coindex as a critical point of the restricted scalar curvature functional $\scal_{\mathsmaller N}\!|_{\eM^{\fG}_{{\mathsmaller N},1}}$, {\it i.e.\ }the number of negative eigenvalue of the linear map $\diff\big(\Ric^0_{\mathsmaller N}|_{\eM^{\fG}_{{\mathsmaller N},1}}\big)\big|_{\bar{P}_{\gn}}$.
\end{definition} 

We refer to \cite{Lau,LauW} for a detailed treatment on stability and non-degeneracy of invariant Einstein metrics on homogeneous spaces. 

\subsection{Ancient solutions to the Ricci flow} \label{subsect:ancRF} \hfill \par

We recall that a solution to the {\it Ricci flow} on $M$ is a smooth $1$-parameter family of metrics that evolve in the direction of their Ricci tensors. By diffeomorphism invariance of the Ricci tensor, isometries are preserved by the Ricci flow, and hence one can restrict it to a dynamical system on the space of $\fG$-invariant metrics $\eM^{\fG}_{\mathsmaller M}$, {\it i.e.\ }
$$
P'(t) = -2\Ric_{\mathsmaller M}(P(t)) \,\, , \quad P(0) = P_{\zero} \,\, .
$$
If $P_{\zero} \in \eM^{\fG}_{{\mathsmaller M},1}$, then the {\it normalized Ricci flow} on $M$ starting at $P_{\zero}$ takes the form
$$
\tilde{P}'(t) = -2\Ric^0_{\mathsmaller M}(\tilde{P}(t)) \,\, , \quad \tilde{P}(0) = P_{\zero}
$$
where the traceless Ricci tensor has been defined in \eqref{eq:tracelessRic}. It is well known that the normalized Ricci flow preserves the submanifold $\eM^{\fG}_{{\mathsmaller M},1}$ and that it is equivalent to the Ricci flow up to rescaling and time reparametrization. Moreover, by \eqref{diffscal}, the normalized Ricci flow coincides, up to a positive constant, with the $L^2$-gradient flow of the restricted scalar curvature functional on $\eM^{\fG}_{{\mathsmaller M},1}$. \smallskip

In \cite{BWZ}, the authors studied the global behaviour of the restricted scalar curvature functional on $\eM^{\fG}_{{\mathsmaller M},1}$ in order to prove the existence of Einstein metrics using variational techniques. In particular, the authors proved that for any $\e >0$, the scalar curvature functional satisfies the {\it Palais-Smale compactness condition} on the set
$$
(\eM^{\fG}_{{\mathsmaller M},1})_\e \= \{ P \in \eM^{\fG}_{{\mathsmaller M},1} : \scal_{\mathsmaller M}(P)>\e \} \,\, ,
$$
that is, if $(P^{(n)}) \subset \eM^{\fG}_{{\mathsmaller M},1}$ is a sequence with
$$
\scal_{\mathsmaller M}(P^{(n)}) \to \e \quad \text{ and } \quad \big\langle\Ric^0_{\mathsmaller M}(P^{(n)}),\Ric^0_{\mathsmaller M}(P^{(n)})\big\rangle_{P^{(n)}} \to 0 \quad \text{ as } n \to +\infty \,\, ,
$$
then there exists a subsequence of $(P^{(n)})$ converging in the $\cC^{\infty}$-topology to an Einstein metric $P^{(\infty)} \in \eM^{\fG}_{{\mathsmaller M},1}$, as $n \to +\infty$, with $\scal_{\mathsmaller M}(P^{(\infty)})=\e$. In general, the Palais-Smale compactness condition does not hold on the full space $\eM^{\fG}_{{\mathsmaller M},1}$ due to the existence of the so called {\it 0-Palais-Smale sequences}, which are $(P^{(n)}) \subset \eM^{\fG}_{{\mathsmaller M},1}$ such that $\scal_{\mathsmaller M}(P^{(n)}) \to 0$ and $\big\langle\Ric^0_{\mathsmaller M}(P^{(n)}),\Ric^0_{\mathsmaller M}(P^{(n)})\big\rangle_{P^{(n)}} \to 0$ as $n \to +\infty $. Notice that such sequences cannot admit convergent subsequences since $M$ is not a torus. In fact, the limit of any convergent subsequence would be a Ricci-flat, and hence flat (see \cite{AK}), $\fG$-invariant metric. By \cite[Thm.\ 2.1]{BWZ}, the existence of such a solution implies that $\fG^{\zero}/\fH^{\zero}$ is the total space of a homogeneous torus bundle, where $\fG^{\zero}$ (resp.\ $\fH^{\zero}$) denotes the identity component of $\fG$ (resp.\ $\fH$). More precisely, since $0$-Palais-Smale sequences have bounded sectional curvature by the Gap theorem \cite{BLS}, by \cite{Ped1} we know that the sum of the eigenspaces associated to the shrinking eigenvalues of any $0$-Palais-Smale sequence converges to a reductive complement of $\gh$ into a toral $\fH$-subalgebra $\gk$ of $\gg$ and that such sequences collapse along the fibers of the induced (locally) homogeneous torus fibration \eqref{torusfib} while asymptotically approaching, in a precise sense, a $\gk$-submersion metric. \smallskip

Now let $P(t)$ be the solution to the Ricci flow starting from $P_{\zero} \in \eM^{\fG}_{{\mathsmaller M},1}$ and $\tilde{P}(t)$ the corresponding solution to the normalized Ricci flow. We recall that $P(t)$ (resp.\ $\tilde{P}(t)$) is said to be {\it ancient} if it exists on the time interval $(-\infty, 0]$. It is a well known consequence of the maximum principle that if $P(t)$ is ancient, then it must have {monotonic} non-negative scalar curvature (see {\it e.g.\ }\cite[p.\ 102]{CLN06}). Since the two flows are equivalent up to rescaling and time reparametrization, the same is true for the solution $\tilde{P}(t)$. Furthermore, by \cite{S}, $P(t)$ is ancient if and only if $\tilde{P}(t)$ is ancient. In particular there are exactly two possibilities for the behaviour of the normalized Ricci flow as $t \to -\infty$.

The first possibility is that there exists an $\e>0$ such that $\scal_{\mathsmaller M}(\tilde{P}(t))>\e$ for any $t\leq0$, in which case $\tilde{P}(t)$ (and hence $P(t)$) is {\it non-collapsed} and, by \cite[Thm.\ 5.2]{BLS}, $\tilde{P}(t)$ converges to an Einstein metric as $t\to-\infty$. Since the traceless Ricci tensor is the negative $L^2$-gradient of the functional $\scal_{\mathsmaller M}|_{\eM^{\fG}_{{\mathsmaller M},1}}$, such ancient solutions are known to exist whenever $M$ admits a {\it $\fG$-unstable}, $\fG$-invariant Einstein metric (see {\it e.g.\ }\cite{AC,BLS}). The second possibility is that $\scal_{\mathsmaller M}(\tilde{P}(t)) \to 0$ as $t\to-\infty$. In this case, one can always find a sequence of times $t^{(n)} \to -\infty$ such that $P(t^{(n)})$ is a $0$-Palais-Smale sequence and so $\tilde{P}(t)$ (and hence $P(t)$) is {\it collapsed}. {Indeed, for the sake of the reader, we recall the following

\begin{rem} \label{rem:coll}
A 1-parameter family $\{P(t)\}_{t \in I}$ of $\fG$-invariant metrics, $I \subset \bR$ an interval, is said to be {\it non-collapsed} if there exists $\d>0$ such that
$$
\inj(P(t))\big(|\Rm_{\mathsmaller M}(P(t))|_{\mathsmaller{P(t)}}\big)^{\frac12} \geq \d \quad \text{ for any } t \in I \,\, ,
$$
where $\inj(P)$ denotes the injectivity radius of the metric $P$ at the origin $e\fH$ and $|\cdot|_{\mathsmaller{P}}$ denotes the norm on $\gm$, and hence on the tensor space over $\gm$, induced by $P$. Accordingly, $\{P(t)\}_{t \in I}$ is said to be {\it collapsed} if it is not non-collapsed, {\it i.e.\ }if there exists a sequence $(t^{(n)}) \subset I$ such that
$$
\inj(P(t^{(n)}))\big(|\Rm_{\mathsmaller M}(P(t^{(n)}))|_{\mathsmaller{P(t^{(n)})}}\big)^{\frac12} \to 0 \quad \text{ as } n \to +\infty \,\, .
$$
These properties are invariant under time-depending rescaling and time reparametrization.
\end{rem}

\noindent We also recall that, by \cite{BLS}, the following result holds true (see \cite[Rem.\ 5.3]{BLS}).

\begin{proposition} \label{prop:coll-0PS}
Let $P(t)$ be an ancient solution to the homogeneous Ricci flow on $M=\fG/\fH$ starting from $P_{\zero} \in \eM^{\fG}_{{\mathsmaller M},1}$ and $\tilde{P}(t)$ the corresponding solution to the normalized Ricci flow. Then, $P(t)$ is collapsed if and only if $\scal_{\mathsmaller M}(\tilde{P}(t)) \to 0$ as $t\to-\infty$.
\end{proposition}

\begin{proof}
By \cite[Thm.\ 5.2]{BLS}, it follows that $P(t)$ is non-collapsed if and only if for any sequence $t^{(n)} \to -\infty$ there exists a subsequence $(t^{(n_i)}) \subset (t^{(n)})$ such that $|t^{(n_i)}|^{-1}P(t^{(n_i)})$ converges in the $\cC^{\infty}$-topology to a limit Einstein metric on $M$ as $i \to +\infty$. Moreover, since $\tilde{P}(t)$ coincides, up to time reparametrization, to the volume-normalized family $\det(P(t))^{-\frac1m}P(t)$, it follows that: $P(t)$ is non-collapsed if and only if for any sequence $t^{(n)} \to -\infty$ there exists a subsequence $(t^{(n_i)}) \subset (t^{(n)})$ such that $\tilde{P}(t^{(n_i)})$ converges in the $\cC^{\infty}$-topology to a limit $\fG$-invariant metric in $\eM^{\fG}_{{\mathsmaller M},1}$ as $i \to +\infty$. This concludes the proof.
\end{proof}

Notice that, as a byproduct of Proposition \ref{prop:coll-0PS} and \cite{Ped1}, $M$ admits a collapsed ancient solution to the Ricci flow only if it is the total space of a homogenous torus bundle (see also \cite[Rem.\ 5.3]{BLS}).
}

\medskip
\section{The projected Ricci flow} \label{sect:projRF}
\setcounter{equation} 0

In this section, we introduce two important tools that will be fundamental for the proof of our main results, namely the space of {\it generalized submersion metrics} and the {\it projected Ricci tensor}. In the following, we consider a compact homogeneous space $M=\fG/\fH$ and we use the same notation introduced in Section \ref{sect:prel}.

\subsection{The space of generalized submersion metrics} \hfill \par

Consider a maximal toral $\fH$-subalgebra $\gk$ of $\gg$ and the associated homogeneous torus fibration \eqref{torusfib}. We introduce the space of {\it generalized $\gk$-submersion metrics on $M$} as
\begin{equation} \label{def:gensubmetrics}
\widehat{\eM^{\fG}_{\mathsmaller M}(\gk)} \= \Sym(\gt,Q_{\gt})^{\Ad(\fH)} \oplus \Sym_+(\gn,Q_{\gn})^{\Ad(\fK)} \,\, ,
\end{equation}
{\it i.e.\ }we allow the metric on $\gt$ to be degenerate, and we prove the following crucial result.

\begin{prop} \label{smoothextension}
The Ricci curvature $\Ric_{\mathsmaller M}$ can be extended analytically to the space $\widehat{\eM^{\fG}_{\mathsmaller M}(\gk)}$ of ge\-ne\-ra\-li\-zed $\gk$-submersion metrics on $M$. 
\end{prop}

\begin{proof}
We write $P=P_{\gt} \oplus P_{\gn}$ for any $P \in \eM^{\fG}_{\mathsmaller M}(\gk)$ and we observe that $[\ad(T),P_{\gn}](X)=0$ for any $T \in \gt$, $X \in \gn$. Hence, since $Q$ is $\Ad(\fG)$-invariant, a straightforward computation shows that the tensor $S_{\mathsmaller M}(P)$ defined by \eqref{eq:SM} is explicitly given by
\begin{equation} \begin{aligned} \label{Ssub}
S_{\mathsmaller M}(P)(T).\tilde{T} &= 0 \,\, , \\
S_{\mathsmaller M}(P)(T).Y &= -\ad(T).Y +\tfrac12P_{\gn}^{-1}.\ad(P_{\gt}.T).Y \,\, , \\
S_{\mathsmaller M}(P)(X).\tilde{T} &= -\tfrac12P_{\gn}^{-1}.\ad(X).P_{\gt}.\tilde{T} \,\, , \\
S_{\mathsmaller M}(P)(X).Y &= -\tfrac12\pi_{\gm}.\ad(X).Y -\tfrac12P_{\gn}^{-1}.\pi_{\gn}.(\ad(X).P_{\gn}-\ad(P_{\gn}.X)).Y \,\, ,
\end{aligned} \end{equation}
where $X,Y \in \gn$ and $T,\tilde{T} \in \gt$. Here, we denote by $\pi_{\gm}: \gg \to \gm$ and $\pi_{\gn} : \gg \to \gn$ the $Q$-orthogonal projections onto $\gm$ and $\gn$, respectively. In particular, \eqref{Ssub} implies that $S_{\mathsmaller M}(P)$ can be defined for any generalized metric $P \in \widehat{\eM^{\fG}_{\mathsmaller M}(\gk)}$ and that it depends analytically on $P$. Therefore, formulas \eqref{eq:Rm} and \eqref{eq:Ric} can be used to define $\Rm_{\mathsmaller M}(P)$ and $\Ric_{\mathsmaller M}(P)$ for any $P \in \widehat{\eM^{\fG}_{\mathsmaller M}(\gk)}$.
\end{proof}

Moreover, by using Schur's Lemma, we get

\begin{lemma}
For any $P \in \widehat{\eM^{\fG}_{\mathsmaller M}(\gk)}$, it holds that
\begin{equation} \label{splittingRic}
\Ric_{\mathsmaller M}(P) \in \Sym(\gt,Q_{\gt})^{\Ad(\fH)} \oplus \Sym(\gn,Q_{\gn})^{\Ad(\fK)} \,\, .
\end{equation}
\end{lemma}
\begin{proof}
Notice that, by hypothesis, the submodule $\gt$ is $\Ad(\fK)$-invariant and the representation $\Ad(\fK^{\zero})|_{\gt}$ is trivial. Moreover, by Lemma \ref{lemma:almtriv}, $\gn$ does not contain any $\Ad(\fK)$-invariant submodule on which $\Ad(\fK^{\zero})$ acts trivially.
Fix now $P \in \widehat{\eM^{\fG}_{\mathsmaller M}(\gk)}$ and notice that, since $\fK =\fH \fK^{\zero}$, both $P$ and the decomposition \eqref{eq:dec} are $\Ad(\fK)$-invariant. By \eqref{eq:SM} it follows that $S_{\mathsmaller M}(P)$ is $\Ad(\fK)$-invariant and so $\Ric_{\mathsmaller M}(P)$ is $\Ad(\fK)$-invariant as well. Therefore, the claim follows from Schur's Lemma.
\end{proof}

We are going to use \eqref{Ssub} to compute the differential of the tensor $S_{\mathsmaller M}$ defined in \eqref{eq:SM}. In order to do this, fix a generalized metric $P \in \widehat{\eM^{\fG}_{\mathsmaller M}(\gk)}$ and a tangent direction $B \in T_P\widehat{\eM^{\fG}_{\mathsmaller M}(\gk)}$. Since
\begin{equation} \label{diffinverse}
\tfrac{\diff}{\diff s} (P_{\gn} +sB_{\gn})^{-1} \big|_{s=0} = -P_{\gn}^{-1}.B_{\gn}.P_{\gn}^{-1} \,\, ,
\end{equation}
it follows that the differential $\diff S_{\mathsmaller M}|_{P}(B)$ at $P$ in the direction of $B$ is given by
\begin{equation} \begin{aligned} \label{diffSsub}
\diff S_{\mathsmaller M}|_{P}(B)(T).\tilde{T} &= 0 \,\, , \\
\diff S_{\mathsmaller M}|_{P}(B)(T).Y &= -\tfrac12P_{\gn}^{-1}.B_{\gn}.P_{\gn}^{-1}.\ad(P_{\gt}.T).Y +\tfrac12P_{\gn}^{-1}.\ad(B_{\gt}.T).Y \,\, , \\
\diff S_{\mathsmaller M}|_{P}(B)(X).\tilde{T} &= \frac12P_{\gn}^{-1}.B_{\gn}.P_{\gn}^{-1}.\ad(X).P_{\gt}.\tilde{T} -\tfrac12P_{\gn}^{-1}.\ad(X).B_{\gt}.\tilde{T} \,\, , \\
\diff S_{\mathsmaller M}|_{P}(B)(X).Y &= \tfrac12P_{\gn}^{-1}.B_{\gn}.P_{\gn}^{-1}.\pi_{\gn}.(\ad(X).P_{\gn}-\ad(P_{\gn}.X)).Y \\
&\quad -\tfrac12P_{\gn}^{-1}.\pi_{\gn}.(\ad(X).B_{\gn}-\ad(B_{\gn}.X)).Y \,\, ,
\end{aligned} \end{equation}
where $X,Y\in\mathfrak{n}$ and $T,\tilde{T}\in\mathfrak{t}$. Moreover, by differentiating \eqref{eq:Rm} and \eqref{eq:Ric} at $P$ in the direction of $B$, we get
\begin{equation} \begin{aligned} \label{diffRmRic}
&\diff \Rm_{\mathsmaller M}\!|_P(B)(V_1,V_2) = -[\diff S_{\mathsmaller M}|_{P}(B)(V_1),S_{\mathsmaller M}(P)(V_2)] -[S_{\mathsmaller M}(P)(V_1),\diff S_{\mathsmaller M}|_{P}(B)(V_2)] \\
&\qquad\qquad\qquad\qquad\qquad\quad -\diff S_{\mathsmaller M}|_{P}(B)([V_1,V_2]_{\gm}) \,\, , \\
&Q(\diff \Ric_{\mathsmaller M}\!|_P(B).V_1,V_2) = \Tr(\gm \ni Z \mapsto \diff \Rm_{\mathsmaller M}\!|_P(B)(V_1,Z).V_2) \,\, .
\end{aligned} \end{equation}
Therefore, we obtain the following.

\begin{prop} \label{propRic(0P)}
Fix a metric on the base space $P_{\gn} \in \eM^{\fG}_{\mathsmaller N}$. Then, the extended Ricci curvature satisfies 
\begin{equation} \begin{aligned} \label{Ric(0P)1}
\Ric_{\mathsmaller M}(0\oplus P_{\gn}) &= 0 \oplus \Ric_{\mathsmaller N}(P_{\gn}) \,\, , \\
\diff \Ric_{\mathsmaller M}\!|_{0\oplus P_{\gn}}(0\oplus B_{\gn}) &= 0 \oplus \diff \Ric_{\mathsmaller N}\!|_{P_{\gn}}(B_{\gn})
\end{aligned} \end{equation}
for any horizontal direction $B_{\gn} \in \Sym(\gn,Q_{\gn})^{\Ad(\fK)}$, and
\begin{equation} \label{Ric(0P)2}
\diff \Ric_{\mathsmaller M}\!|_{0\oplus P_{\gn}}(B_{\gt}\oplus 0).T = 0
\end{equation}
for any vertical direction $B_{\gt} \in \Sym(\gt,Q_{\gt})^{\Ad(\fH)}$ and for any $T \in \gt$.
\end{prop}

\begin{proof}
Fix $B_{\gn} \in \Sym(\gn,Q_{\gn})^{\Ad(\fK)}$ and let $X,Y,Z \in \gn$, $T,\tilde{T} \in \gt$. Then, from \eqref{Ssub} and \eqref{diffSsub}, it follows that the operators $S_{\mathsmaller M}(0\oplus P_{\gn})$ and $\diff S_{\mathsmaller M}|_{0\oplus P_{\gn}}(0\oplus B_{\gn})$ satisfy
\begin{equation} \begin{aligned} \label{SM(0P)}
&S_{\mathsmaller M}(0\oplus P_{\gn})(T).\tilde{T} = 0 \,\, , \quad S_{\mathsmaller M}(0\oplus P_{\gn})(T).Y = -\ad(T).Y \,\, , \quad S_{\mathsmaller M}(0\oplus P_{\gn})(X).\tilde{T} = 0 \,\, , \\
&S_{\mathsmaller M}(0\oplus P_{\gn})(X).Y = -\tfrac12\pi_{\gm}.\ad(X).Y -\tfrac12P_{\gn}^{-1}.\pi_{\gn}.(\ad(X).P_{\gn}-\ad(P_{\gn}.X)).Y
\end{aligned} \end{equation}
and
\begin{equation} \begin{aligned} \label{diffSM(0P)1}
&\diff S_{\mathsmaller M}|_{0\oplus P_{\gn}}(0\oplus B_{\gn})(T).\tilde{T} = 0 \,\, , \quad \diff S_{\mathsmaller M}|_{0\oplus P_{\gn}}(0\oplus B_{\gn})(T).Y = 0 \,\, , \quad \diff S_{\mathsmaller M}|_{0\oplus P_{\gn}}(0\oplus B_{\gn})(X).\tilde{T} = 0 \,\, , \\
&\diff S_{\mathsmaller M}|_{0\oplus P_{\gn}}(0\oplus B_{\gn})(X).Y = \tfrac12P_{\gn}^{-1}.B_{\gn}.P_{\gn}^{-1}.\pi_{\gn}.(\ad(X).P_{\gn}-\ad(P_{\gn}.X)).Y \\
&\qquad\qquad\qquad\qquad\qquad\qquad -\tfrac12P_{\gn}^{-1}.\pi_{\gn}.(\ad(X).B_{\gn}-\ad(B_{\gn}.X)).Y \,\, .
\end{aligned} \end{equation}
On the other hand, by using \eqref{eq:SM} and \eqref{diffinverse}, it follows that the operators $S_{\mathsmaller N}(P_{\gn})$ and $\diff S_{\mathsmaller N}|_{P_{\gn}}(B_{\gn})$ satisfy
\begin{equation} \begin{aligned} \label{SN(0P)}
&S_{\mathsmaller N}(P_{\gn})(X).Y = -\tfrac12\pi_{\gn}.\ad(X).Y -\tfrac12P_{\gn}^{-1}.\pi_{\gn}.(\ad(X).P_{\gn}-\ad(P_{\gn}.X)).Y \,\, , \\
&\diff S_{\mathsmaller N}|_{P_{\gn}}(B_{\gn})(X).Y = \tfrac12P_{\gn}^{-1}.B_{\gn}.P_{\gn}^{-1}.\pi_{\gn}.(\ad(X).P_{\gn}-\ad(P_{\gn}.X)).Y \\
&\qquad\qquad\qquad\qquad\quad -\tfrac12P_{\gn}^{-1}.\pi_{\gn}.(\ad(X).B_{\gn}-\ad(B_{\gn}.X)).Y \,\, .
\end{aligned} \end{equation}
A straightforward computation based on \eqref{eq:Rm}, \eqref{diffRmRic}, \eqref{SM(0P)} and \eqref{diffSM(0P)1} shows that
$$
\Rm_{\mathsmaller M}(0\oplus P_{\gn})(T,\cdot).\tilde{T} = \diff \Rm_{\mathsmaller M}\!|_{0\oplus P_{\gn}}(0\oplus B_{\gn})(T,\cdot).\tilde{T} = 0
$$
and so, by using \eqref{eq:Ric} and \eqref{diffRmRic}, we get
$$
\Ric_{\mathsmaller M}(0\oplus P_{\gn})(T) \in \gn \quad \text{ and } \quad \diff \Ric_{\mathsmaller M}\!|_{0\oplus P_{\gn}}(0\oplus B_{\gn})(T) \in \gn \,\, .
$$
Therefore, \eqref{splittingRic} implies that
\begin{equation} \label{ext1}
\Ric_{\mathsmaller M}(0\oplus P_{\gn})(T) = \diff \Ric_{\mathsmaller M}\!|_{0\oplus P_{\gn}}(0\oplus B_{\gn})(T) = 0 \,\, .
\end{equation}
Again, using \eqref{eq:Rm}, \eqref{diffRmRic}, \eqref{SM(0P)} and \eqref{diffSM(0P)1} one can directly check that
$$
\Rm_{\mathsmaller M}(0\oplus P_{\gn})(X,\cdot).\tilde{T} = \diff \Rm_{\mathsmaller M}\!|_{0\oplus P_{\gn}}(0\oplus B_{\gn})(X,\cdot).\tilde{T} = 0
$$
and so \eqref{eq:Ric} and \eqref{diffRmRic} imply that
\begin{equation} \label{ext2}
\Ric_{\mathsmaller M}(0\oplus P_{\gn})(X) \in \gn \quad \text{ and } \quad \diff \Ric_{\mathsmaller M}\!|_{0\oplus P_{\gn}}(0\oplus B_{\gn})(X) \in \gn \,\, .
\end{equation}
Finally, another direct computation based on \eqref{eq:Rm}, \eqref{diffRmRic}, \eqref{SM(0P)}, \eqref{diffSM(0P)1} and \eqref{SN(0P)} shows that
\begin{equation} \begin{aligned} \label{ext3}
\pi_{\gn}(\Rm_{\mathsmaller M}(0\oplus P_{\gn})(X,Y).Z) &= \Rm_{\mathsmaller N}(P_{\gn})(X,Y).Z \,\, , \\
\pi_{\gn}(\diff \Rm_{\mathsmaller M}\!|_{0\oplus P_{\gn}}(0\oplus B_{\gn})(X,Y).Z) &= \diff \Rm_{\mathsmaller N}\!|_{P_{\gn}}(B_{\gn})(X,Y).Z \,\, .
\end{aligned} \end{equation}
Notice now that \eqref{Ric(0P)1} follows from \eqref{ext1}, \eqref{ext2} and \eqref{ext3}. In order to prove \eqref{Ric(0P)2}, fix $B_{\gt} \in \Sym(\gt,Q_{\gt})^{\Ad(\fH)}$ and observe that, from \eqref{diffSsub}, it follows that the operator $\diff S_{\mathsmaller M}|_{0\oplus P_{\gn}}(B_{\gt}\oplus 0)$ satisfies
\begin{equation} \begin{aligned} \label{diffSM(0P)2}
\diff S_{\mathsmaller M}|_{0\oplus P_{\gn}}(B_{\gt}\oplus 0)(T).\tilde{T} &= 0 \,\, , \\
\diff S_{\mathsmaller M}|_{0\oplus P_{\gn}}(B_{\gt}\oplus 0)(T).Y &= +\tfrac12P_{\gn}^{-1}.\ad(B_{\gt}.T).Y \,\, , \\
\diff S_{\mathsmaller M}|_{0\oplus P_{\gn}}(B_{\gt}\oplus 0)(X).\tilde{T} &= -\tfrac12P_{\gn}^{-1}.\ad(X).B_{\gt}.\tilde{T} \,\, , \\
\diff S_{\mathsmaller M}|_{0\oplus P_{\gn}}(B_{\gt}\oplus 0)(X).Y &= 0 \,\, .
\end{aligned} \end{equation}
Again, by using \eqref{eq:Rm}, \eqref{eq:Ric}, \eqref{diffRmRic}, \eqref{SM(0P)} and \eqref{diffSM(0P)2}, one can show that
$$
\diff \Ric_{\mathsmaller M}\!|_{0\oplus P_{\gn}}(B_{\gt}\oplus 0)(T) \in \gn
$$
and so, using \eqref{splittingRic}, we get \eqref{Ric(0P)2}.
\end{proof}

\subsection{The $\bar{P}_{\gn}$-projected Ricci tensor} \hfill \par

Fix a unit volume Einstein metric $\bar{P}_{\gn} \in \eM^{\fG}_{{\mathsmaller N},1}$ on $N$, {\it i.e.\ }$\Ric_{\mathsmaller N}(\bar{P}_{\gn}) = \l\,\bar{P}_{\gn}$ for some $\l \in \bR$. Since $N$ is compact, Bochner's Theorem implies that $\l$ is non-negative (see \cite{Boc}). Moreover, since $M=\fG/\fH$ is not a torus, then also $N$ is not a torus and so $\l >0$.

We introduce the Euclidean inner product $\langle\!\langle\cdot,\cdot\rangle\!\rangle^{\mathsmaller{(\!\bar{P}_{\gn}\!)}}$ on the linear space $\Sym(\gm,Q_{\gm})^{\Ad(\fH)}$ defined by
\begin{equation} \label{prod2}
\langle\!\langle B_1, B_2\rangle\!\rangle^{\mathsmaller{(\!\bar{P}_{\gn}\!)}} \= \dim(N)^{-1}\Tr\!\big((\Id_{\gt} \oplus (\bar{P}_{\gn})^{-1}).B_1.(\Id_{\gt} \oplus (\bar{P}_{\gn})^{-1}).B_2\big)
\end{equation}
and the {\it $\bar{P}_{\gn}$-projected Ricci curvature}
\begin{equation} \begin{gathered} \label{def:R}
\cR^{\mathsmaller{(\!\bar{P}_{\gn}\!)}}_{\mathsmaller M}: \widehat{\eM^{\fG}_{\mathsmaller M}(\gk)} \to \Sym(\gt,Q_{\gt})^{\Ad(\fH)} \oplus \Sym(\gn,Q_{\gn})^{\Ad(\fK)} \,\, , \\
\cR^{\mathsmaller{(\!\bar{P}_{\gn}\!)}}_{\mathsmaller M}(P) \= \Ric_{\mathsmaller M}(P) -\frac{\langle\!\langle \Ric_{\mathsmaller M}(P), P\rangle\!\rangle^{\mathsmaller{(\!\bar{P}_{\gn}\!)}}}{\langle\!\langle P, P\rangle\!\rangle^{\mathsmaller{(\!\bar{P}_{\gn}\!)}}}P \,\, .
\end{gathered} \end{equation}
We remark that, for any $P \in \widehat{\eM^{\fG}_{\mathsmaller M}(\gk)}$, the image $\cR^{\mathsmaller{(\!\bar{P}_{\gn}\!)}}_{\mathsmaller M}(P)$ lies in $\Sym(\gt,Q_{\gt})^{\Ad(\fH)} \oplus \Sym(\gn,Q_{\gn})^{\Ad(\fK)}$ by means of \eqref{splittingRic}. As a consequence of Proposition \ref{propRic(0P)}, we get the following

\begin{corollary}
The $\bar{P}_{\gn}$-projected Ricci curvature $\cR^{\mathsmaller{(\!\bar{P}_{\gn}\!)}}_{\mathsmaller M}$ satisfies
\begin{equation} \begin{aligned} \label{diffR1}
\cR^{\mathsmaller{(\!\bar{P}_{\gn}\!)}}_{\mathsmaller M}(0 \oplus \bar{P}_{\gn}) &= 0 \oplus \Ric_{\mathsmaller N}^0(\bar{P}_{\gn}) \,\, , \\
\diff \cR^{\mathsmaller{(\!\bar{P}_{\gn}\!)}}_{\mathsmaller M}\big|_{0 \oplus \bar{P}_{\gn}}(0 \oplus B_{\gn}) &= 0 \oplus \diff{\Ric_{\mathsmaller N}^0}|_{\bar{P}_{\gn}}(B_{\gn})
\end{aligned} \end{equation}
for any $B_{\gn} \in \Sym(\gn,Q_{\gn})^{\Ad(\fK)}$, and
\begin{equation} \begin{aligned} \label{diffR2}
\diff \cR^{\mathsmaller{(\!\bar{P}_{\gn}\!)}}_{\mathsmaller M}\big|_{0 \oplus \bar{P}_{\gn}}(B_{\gt} \oplus 0).T = -\l B_{\gt}.T
\end{aligned} \end{equation}
for any $B_{\gt} \in \Sym(\gt,Q_{\gt})^{\Ad(\fH)}$, $T \in \gt$.
\end{corollary}

\begin{proof}
Notice that \eqref{diffR1} follows from a direct computation based on \eqref{eq:SM}, \eqref{eq:Rm}, \eqref{eq:Ric} and \eqref{Ric(0P)1}. Moreover, from \eqref{Ric(0P)1} and \eqref{Ric(0P)2}, we get
\begin{align*}
\diff \cR^{\mathsmaller{(\!\bar{P}_{\gn}\!)}}_{\mathsmaller M}|_{0 \oplus \bar{P}_{\gn}}(B_{\gt} \oplus 0).T &= -\diff\bigg(\frac{\langle\!\langle \Ric_{\mathsmaller M}(P), P\rangle\!\rangle^{\mathsmaller{(\!\bar{P}_{\gn}\!)}}}{\langle\!\langle P, P\rangle\!\rangle^{\mathsmaller{(\!\bar{P}_{\gn}\!)}}}P\bigg)\bigg|_{0 \oplus \bar{P}_{\gn}}(B_{\gt} \oplus 0).T \\
&= -\diff \bigg(\frac{\langle\!\langle \Ric_{\mathsmaller M}(P), P\rangle\!\rangle^{\mathsmaller{(\!\bar{P}_{\gn}\!)}}}{\langle\!\langle P, P\rangle\!\rangle^{\mathsmaller{(\!\bar{P}_{\gn}\!)}}}\bigg)\bigg|_{0 \oplus \bar{P}_{\gn}}(B_{\gt} \oplus 0) \cdot (0 \oplus \bar{P}_{\gn}).T \\
&\qquad -\bigg(\frac{\langle\!\langle \Ric_{\mathsmaller M}(0\oplus \bar{P}_{\gn}), 0\oplus \bar{P}_{\gn}\rangle\!\rangle^{\mathsmaller{(\!\bar{P}_{\gn}\!)}}}{\langle\!\langle 0\oplus \bar{P}_{\gn}, 0\oplus \bar{P}_{\gn}\rangle\!\rangle^{\mathsmaller{(\!\bar{P}_{\gn}\!)}}}\bigg) \cdot (B_{\gt} \oplus 0).T \\
&= 0 -\frac{\scal_{\mathsmaller N}(\bar{P}_{\gn})}{\dim(N)} B_{\gt}.T \\
&= -\l B_{\gt}.T
\end{align*}
for any $B_{\gt} \in \Sym(\gt,Q_{\gt})^{\Ad(\fH)}$ and $T \in \gt$, which proves \eqref{diffR2}.
\end{proof}

In virtue of Proposition \ref{smoothextension} and \eqref{propRic(0P)}, the Ricci flow preserves the subspace $\eM^{\fG}_{\mathsmaller M}(\gk)$ of $\gk$-submersion metrics and can be extended to the larger space $\widehat{\eM^{\fG}_{\mathsmaller M}(\gk)}$ of generalized $\gk$-submersion metrics. Moreover, since the Ricci curvature $\Ric_{\mathsmaller M}$ is scale invariant, we may project the Ricci flow onto the unit sphere
$$
\S^{\mathsmaller{(\!\bar{P}_{\gn}\!)}} \= \Big\{P \in \widehat{\eM^{\fG}_{\mathsmaller M}(\gk)} : \langle\!\langle P, P\,\rangle\!\rangle^{\mathsmaller{(\!\bar{P}_{\gn}\!)}}= 1 \Big\}
$$
of $\widehat{\eM^{\fG}_{\mathsmaller M}(\gk)}$ with respect to the inner product $\langle\!\langle\cdot,\cdot\rangle\!\rangle^{\mathsmaller{(\!\bar{P}_{\gn}\!)}}$. Hence up to rescaling, the Ricci flow is equivalent to the flow on $
\S^{\mathsmaller{(\!\bar{P}_{\gn}\!)}}$ defined by
\begin{equation} \label{projRF}
P'(t) = -2\cR^{\mathsmaller{(\!\bar{P}_{\gn}\!)}}_{\mathsmaller M}(P(t)) \,\, ,
\end{equation}
which we call the {\it $\bar{P}_{\gn}$-projected Ricci flow}.

\medskip
\section{Proof of Theorem \ref{main-A}} \label{sect:mainproof}
\setcounter{equation} 0

This section is devoted to the proof of our main result. In the following, we consider a compact homogeneous space $M=\fG/\fH$, a fixed maximal toral $\fH$-subalgebra $\gk$ of $\gg$ and we use the same notation as in Section \ref{sect:prel} and Section \ref{sect:projRF}.

\subsection{Two preparatory results} \hfill \par

Take a sequence $(P^{(n)}) \subset \eM^{\fG}_{\mathsmaller M}(\gk)$ of $\gk$-submersion metrics $P^{(n)} = P^{(n)}_{\gt} \oplus P^{(n)}_{\gn}$ such that $P^{(n)}_{\gt} \to 0$ and $P^{(n)}_{\gn} \to P^{(\infty)}_{\gn} \in \eM^{\fG}_{\mathsmaller N}$ as $n \to +\infty$. The first result that we need for proving Theorem \ref{main-A} is the following.

\begin{prop}
The scalar curvature of $P^{(n)}$ converges to the scalar curvature of $P^{(\infty)}_{\gn}$, that is
\begin{equation} \label{limitscal}
\scal_{\mathsmaller M}(P^{(n)}) \to \scal_{\mathsmaller N}(P^{(\infty)}_{\gn}) \quad \text{ as $n \to +\infty$ .}
\end{equation}
\end{prop}

\begin{proof}
Since the fibers of \eqref{torusfib} are totally geodesic and flat along the sequence, by O'Neill's Formula (see \cite[Eq.\ (9.37)]{Bes}) we get
$$
\scal_{\mathsmaller M}(P^{(n)}) = \scal_{\mathsmaller N}(P^{(n)}_{\gn}) - \big(\big|A^{(n)}\big|_{\mathsmaller{P^{(n)}}}\big)^2 \,\, ,
$$
where $A^{(n)}: \gm \otimes \gm \to \gm$ is the {\it O'Neill's integrability tensor} for the Riemannian submersion induced by \eqref{torusfib} and the metric $P^{(n)}$.

Since the scalar curvature functional is continuous, it follows that $\scal_{\mathsmaller N}(P^{(n)}_{\gn}) \to \scal_{\mathsmaller N}(P^{(\infty)}_{\gn})$. Therefore, in order to prove \eqref{limitscal}, it is sufficient to show that $\big|A^{(n)}\big|_{\mathsmaller{P^{(n)}}} \to 0$ as $n \to +\infty$.

For any $n \in \bN$, we consider a $Q_{\gm}$-orthogonal, $\Ad(\fH)$-invariant decomposition into irreducible modules
\begin{equation} \label{decm}
\gm = \gm_1^{(n)} + {\dots} + \gm_{\ell}^{(n)}
\end{equation}
with respect to which $P^{(n)}$ is diagonal, {\it i.e.\ }
$$
P^{(n)} = x_1^{(n)} \Id_{\gm_1^{(n)}} \oplus {\dots} \oplus x_{\ell}^{(n)} \Id_{\gm_{\ell}^{(n)}} \,\, , \quad x_k^{(n)}>0 \,\, \text{ for any $1 \leq k \leq \ell$} \,\, .
$$
By hypothesis, we can assume that: \begin{itemize}
\item[$\bcdot$] the dimension $m_i \= \dim(\gm_i^{(n)})$ is constant along the sequence for any $1 \leq i \leq \ell$;
\item[$\bcdot$] the decomposition \eqref{decm} converges to a well defined $\Ad(\fH)$-invariant, irreducible, limit decomposition $\gm = \gm_1^{(\infty)} + {\dots} + \gm_{\ell}^{(\infty)}$;
\item[$\bcdot$] there exists $1 \leq r \leq \ell$ such that $$
\gt = \gm_1^{(n)} + {\dots} + \gm_{r}^{(n)} \,\, , \quad \gn = \gm_{r+1}^{(n)} + {\dots} + \gm_{\ell}^{(n)} \quad \text{ for any $n \in \bN$ } ;
$$
\item[$\bcdot$] $P^{(\infty)}_{\gn}$ is diagonal with respect to $\gn = \gm_{r+1}^{(\infty)} + {\dots} + \gm_{\ell}^{(\infty)}$, {\it i.e.\ }
$$
P^{(\infty)}_{\gn} = x_{r+1}^{(\infty)} \Id_{\gm_{r+1}^{(\infty)}} \oplus {\dots} \oplus x_{\ell}^{(\infty)} \Id_{\gm_{\ell}^{(\infty)}} \,\, , \quad x_j^{(\infty)}>0 \,\, \text{ for any $r+1 \leq j \leq \ell$} \,\, .
$$
\end{itemize}
Notice that, by hypothesis, it follows that $x_i^{(n)} \to 0$ as $n \to +\infty$ for any $1 \leq i \leq r$ and $x_j^{(n)} \to x_j^{(\infty)}$ as $n \to +\infty$ for any $r+1 \leq j \leq \ell$. We consider now a sequence of {\it adapted bases}, {\it i.e.\ }for any $n \in \bN$ we consider a $Q_{\gm}$-orthonormal basis $(e_{\a}^{(n)})_{1 \leq \a \leq m}$ for $\gm$ such that
$$
e_1^{(n)} , {\dots} , e_{m_1}^{(n)} \in \gm_1^{(n)} \,\, , \quad
e_{m_1+1}^{(n)} , {\dots} , e_{m_1+m_2}^{(n)} \in \gm_2^{(n)} \,\, , \quad 
{\dots} \quad , \quad
e_{m_1+{\dots}+m_{\ell-1}+1}^{(n)} , {\dots} , e_m^{(n)} \in \gm_{\ell}^{(n)} \,\, ,
$$
and we define the coefficients
\begin{equation} \label{def:[ijk]}
[ijk]^{(n)} \= \sum_{e_{\a}^{(n)} \in \gm_i^{(n)}} \sum_{e_{\b}^{(n)} \in \gm_j^{(n)}} \sum_{e_{\g}^{(n)} \in \gm_k^{(n)}} Q\big([e_{\a}^{(n)},e_{\b}^{(n)}],e_{\g}^{(n)}\big)^2 \,\, .
\end{equation}
Notice that $[ijk]^{(n)}$ is symmetric in all its entries and does not depend on the choice of $(e_{\a}^{(n)})$. Moreover, we can assume that $(e_{\a}^{(n)})$ converges to a limit adapted basis $(e_{\a}^{(\infty)})$ for $\gm$ and, as a consequence, $[ijk]^{(n)}$ converges to the coefficient $[ijk]^{(\infty)}$ related to the limit decomposition. For more information about the diagonalization of invariant metrics on compact homogeneous spaces, we refer to \cite{WaZ, Boe1}.

For the sake of shortness, we set
$$
A_{ij}^{(n)} \= \sum_{e_{\a}^{(n)} \in \gm_i^{(n)}} \sum_{e_{\b}^{(n)}\in \gm_j^{(n)}} \big(\big|A^{(n)}\big(e_{\a}^{(n)},e_{\b}^{(n)}\big)\big|_{\mathsmaller{P^{(n)}}}\big)^2 \,\, .
$$
Notice that by \cite[Lemma 2]{ON} and \eqref{def:[ijk]}, it follows that
\begin{equation} \label{dim-a1}
A_{j_1j_2}^{(n)} = \tfrac14 \sum_{1 \leq i \leq r} [ij_1j_2]^{(n)}x_i^{(n)} \to 0 \quad \text{ for any $r+1 \leq j_1, j_2 \leq \ell$ .}
\end{equation}
Moreover, since O'Neill's tensor is horizontal (see \cite[p.\ 460]{ON}), it follows that
\begin{equation} \label{dim-a2}
A_{ik}^{(n)} =0 \quad \text{ for any $1 \leq i \leq r$ , $1 \leq k \leq \ell$ .}
\end{equation}
Finally, by \cite[Cor.\ 1]{ON} and \cite[Eq.\ (4.5) and (4.7)]{Ped1}, we obtain
\begin{equation} \label{dim-a31}
A_{ji}^{(n)} = \tfrac14\sum_{1 \leq k \leq \ell}[ijk]^{(n)}\tfrac{x_i^{(n)}}{x_j^{(n)}x_k^{(n)}} +\tfrac14\sum_{1 \leq k \leq \ell}[ijk]^{(n)}\Big(\tfrac{x_j^{(n)}}{x_k^{(n)}}-1\Big)\Big(-2\tfrac{x_i^{(n)}}{x_j^{(n)}}+1+3\tfrac{x_k^{(n)}}{x_j^{(n)}}\Big)\tfrac1{x_i^{(n)}}
\end{equation}
for any $1 \leq i \leq r$, $r+1 \leq j \leq \ell$. Since each $P^{(n)}$ is a $\gk$-submersion metric and $\gt$ is abelian, it follows that
\begin{equation} \begin{gathered} \label{dim-a32}
\, [i_1i_2k]^{(n)} = 0 \quad \text{ for any } \, 1 \leq i_1, i_2 \leq r \, , \,\, 1 \leq k \leq \ell \, , \quad \text{ for any } n \in \bN \,\, , \\
\, [ij_1j_2]^{(n)}\Big(\tfrac{x_{j_2}^{(n)}}{x_{j_1}^{(n)}}-1\Big) = 0 \quad \text{ for any } \, 1 \leq i \leq r \, , \,\, r+1 \leq j_1, j_2 \leq r \, , \quad \text{ for any } n \in \bN \,\, .
\end{gathered} \end{equation}
Therefore, by \eqref{dim-a31} and \eqref{dim-a32} we get
\begin{equation} \label{dim-a3}
A_{ji}^{(n)} = \tfrac14\sum_{r+1 \leq j' \leq \ell}[ijj']^{(n)}\tfrac{x_i^{(n)}}{x_j^{(n)}x_{j'}^{(n)}} \to 0 \quad \text{ for any $1 \leq i \leq r$ , $r+1 \leq j \leq \ell$ }
\end{equation}
and so the claim follows from \eqref{dim-a1}, \eqref{dim-a2} and \eqref{dim-a3}.
\end{proof}

Let us denote now by $\td^{(n)}_{\mathsmaller M}$ the Riemannian distance induced by $P^{(n)}$ on $M$ and by $\td^{(n)}_{\mathsmaller N}$ (resp.\ $\td^{(\infty)}_{\mathsmaller N}$) the Riemannian distance induced by $P^{(n)}_{\gn}$ (resp.\ $P^{(\infty)}_{\gn}$) on $N$. We recall that, since $N$ is compact and $P^{(n)}_{\gn} \to P^{(\infty)}_{\gn}$ in the $\cC^{\infty}$-topology, it follows that the metric spaces $(N,\td^{(n)}_{\mathsmaller N})$ converge to $(N,\td^{(\infty)}_{\mathsmaller N})$ in the {\it Gromov-Hausdorff topology} as $n \to +\infty$ (see {\it e.g.\ }\cite[p.\ 415]{Pet}). For a detailed treatment on Gromov-Hausdorff convergence, we refer to \cite{BBI,Ro}.

\begin{prop} \label{propconvGH}
The sequence of compact metric spaces $(M,\td^{(n)}_{\mathsmaller M})$ converges to $(N,\td^{(\infty)}_{\mathsmaller N})$ in the Gromov-Hausdorff topology as $n \to +\infty$.
\end{prop}

\begin{proof}
In order to prove the statement, it is sufficient to show that
$$
\big|\td^{(n)}_{\mathsmaller M}(a_0\fH,a_1\fH) - \td^{(\infty)}_{\mathsmaller N}(a_0\fK,a_1\fK)\big| \xrightarrow{n \to +\infty} 0 \quad \text{ uniformly in $a_0,a_1 \in \fG$ .}
$$

Fix $a_0,a_1 \in \fG$ and consider for any $n \in \bN$ a $\td^{(n)}_{\mathsmaller N}$-geodesic $\g^{(n)}: [0,1] \to N$ such that $\g^{(n)}(0)=a_0\fK$, $\g^{(n)}(1)=a_1\fK$, which realizes the $\td^{(n)}_{\mathsmaller N}$-distance between $a_0\fK$ and $a_1\fK$. Consider now the horizontal lift $\g^{(n)\uparrow}: [0,1] \to M$ of $\g^{(n)}$ to $M$ starting from $a_0\fH$ and pick $c^{(n)} \in \fT$ such that $\g^{(n)\uparrow}(1) = a_1c^{(n)}\fH$. Since $P^{(n)}$ is a $\gk$-submersion metric, it follows that $\td^{(n)}_{\mathsmaller M}(a_0\fH,a_1c^{(n)}\fH) = \td^{(n)}_{\mathsmaller N}(a_0\fK,a_1\fK)$. Then, by the reverse triangle inequality, we get
\begin{equation} \label{proofGH}
\big|\td^{(n)}_{\mathsmaller M}(a_0\fH,a_1\fH) - \td^{(\infty)}_{\mathsmaller N}(a_0\fK,a_1\fK)\big| \leq \td^{(n)}_{\mathsmaller M}(a_1\fH,a_1c^{(n)}\fH) +\big|\td^{(n)}_{\mathsmaller N}(a_0\fK,a_1\fK) - \td^{(\infty)}_{\mathsmaller N}(a_0\fK,a_1\fK)\big| \,\, .
\end{equation}
Notice now that both the terms on the right hand side of \eqref{proofGH} converge uniformly to $0$ as $n \to +\infty$, and this concludes the proof.
\end{proof}

Let us finally remark that both \eqref{limitscal} and Proposition \ref{propconvGH} hold true for any (not necessarily maximal) toral $\fH$-subalgebra $\gk$.

\subsection{The existence theorem} \label{subsect:existence} \hfill \par

Consider again a unit volume Einstein metric $\bar{P}_{\gn} \in \eM^{\fG}_{{\mathsmaller N},1}$ on $N$ with $\Ric_{\mathsmaller N}(\bar{P}_{\gn}) = \l\,\bar{P}_{\gn}$ for some $\l>0$. We also set
$$
\nu \= \dim\!\big(\Sym(\gt,Q_{\gt})^{\Ad(\fH)}\big) \,\, .
$$
Notice that, if $\fH$ is connected, then $\Ad(\fH)|_{\gt}$ is trivial and so $\nu = \tfrac{d(d+1)}2$, where $d \= \dim(\fT)$. However, in the general case it may happen that $1 \leq \nu < \tfrac{d(d+1)}2$. \smallskip

The main result of this section is the following

\begin{theorem} \label{thm:existence}
If $\bar{P}_{\gn}$ has coindex $q$, then there exists a $(\nu+q-1)$-parameter family of ancient solutions to the $\bar{P}_{\gn}$-projected Ricci flow on $\eM^{\fG}_{\mathsmaller M}(\gk)$ which converge to $0 \oplus \bar{P}_{\gn}$ as $t \to -\infty$ and such that the corresponding solutions to the Ricci flow are ancient and collapsed.
\end{theorem}

\begin{proof}
Let us observe that the $\bar{P}_{\gn}$-projected Ricci tensor \eqref{def:R} is defined on an open neighborhood of $0 \oplus \bar{P}_{\gn}$ inside $\S^{\mathsmaller{(\!\bar{P}_{\gn}\!)}}$. Moreover, from \eqref{prod2} it holds that
\begin{equation} \label{tangSigma}
T_{0 \oplus \bar{P}_{\gn}}\S^{\mathsmaller{(\!\bar{P}_{\gn}\!)}} = \Sym(\gt,Q_{\gt})^{\Ad(\fH)} \oplus T_{\bar{P}_{\gn}}\eM^{\fG}_{{\mathsmaller N},1}
\end{equation}
and, by \eqref{diffR1} and \eqref{diffR2}, it follows that
\begin{equation} \label{linearizationR}
\cR^{\mathsmaller{(\!\bar{P}_{\gn}\!)}}_{\mathsmaller M}(0 \oplus \bar{P}_{\gn}) = 0 \,\, , \quad \diff \cR^{\mathsmaller{(\!\bar{P}_{\gn}\!)}}_{\mathsmaller M}\big|_{0 \oplus \bar{P}_{\gn}} = \left(\begin{array}{c|c}
-\l \Id_{\Sym(\gt,Q_{\gt})^{\Ad(\fH)}} & 0 \\ \hline
* & \diff{\Ric_{\mathsmaller N}^0}|_{\bar{P}_{\gn}}
\end{array}\right) \,\, .
\end{equation}

By \eqref{tangSigma}, \eqref{linearizationR} and the Center Manifold Theorem \cite[p.\ 116]{Per}, it follows that there exists a stable manifold $\widehat{\eW^{\mathsmaller{(\!\bar{P}_{\gn}\!)}}}$ for $\cR^{\mathsmaller{(\!\bar{P}_{\gn}\!)}}_{\mathsmaller M}$ at $0 \oplus \bar{P}_{\gn}$ of dimension $\dim\widehat{\eW^{\mathsmaller{(\!\bar{P}_{\gn}\!)}}}=\nu+q$, where $q$ is the coindex of $\bar{P}_{\gn}$ (see Definition \ref{def:coindex}). We remark that $\widehat{\eW^{\mathsmaller{(\!\bar{P}_{\gn}\!)}}}$ is a submanifold of $\widehat{\eM^{\fG}_{\mathsmaller M}(\gk)}$ and that, being eventually interested in the positive-definite solutions to the Ricci flow, we need to compute the dimension of the manifold $\eW^{\mathsmaller{(\!\bar{P}_{\gn}\!)}} \= \widehat{\eW^{\mathsmaller{(\!\bar{P}_{\gn}\!)}}} \cap \eM^{\fG}_{\mathsmaller M}(\gk)$. For this purpose let us observe that, restricting to the sphere $\S^{\mathsmaller{(\!\bar{P}_{\gn}\!)}}$, the eigenvectors of $\diff \cR^{\mathsmaller{(\!\bar{P}_{\gn}\!)}}_{\mathsmaller M}\big|_{0 \oplus \bar{P}_{\gn}}$ consist of two families of endomorphisms inside $T_{0 \oplus \bar{P}_{\gn}}\S^{\mathsmaller{(\!\bar{P}_{\gn}\!)}}$, namely: \begin{itemize}
\item[$\bcdot$] those coming from the upper left block of \eqref{linearizationR}, spanned by a basis of the form
$$
\eB_1= \big((B_{\gt})_i \oplus (B_{\gn})_i\big) \,\, , \quad 1 \leq i \leq \nu \,\, ;
$$
\item[$\bcdot$] those coming from the lower right block of \eqref{linearizationR}, spanned by a basis of the form
$$
\eB_2= \big(0 \oplus (C_{\gn})_j\big) \,\, , \quad 1 \leq j \leq p-1 \,\, ,
$$
where $p \= \dim \eM^{\fG}_{\mathsmaller N}$.
\end{itemize}
We claim that the endomorphisms $(B_{\gt})_i$ must be linearly independent inside $\Sym(\gt,Q_{\gt})^{\Ad(\fH)}$. If not, then there is a non-trivial linear combination
$$
\sum_i \mu_i\big((B_{\gt})_i \oplus (B_{\gn})_i\big) = 0 \oplus B_{\gn}^{\star} \quad \text{ for some non-zero } B_{\gn}^{\star} \in T_{\bar{P}_{\gn}}\eM^{\fG}_{{\mathsmaller N},1} \,\, .
$$
Since $(C_{\gn})_j$ forms a basis for $T_{\bar{P}_{\gn}}\eM^{\fG}_{{\mathsmaller N},1}$, there is another linear combination
$$
\sum_j \tilde{\mu}_j (C_{\gn})_j = B_{\gn}^{\star} \,\, ,
$$
but this contradicts the fact that $\eB_1 \cup \eB_2$ is a basis for $T_{0 \oplus \bar{P}_{\gn}}\S^{\mathsmaller{(\!\bar{P}_{\gn}\!)}}$. This shows in particular that $\eW^{\mathsmaller{(\!\bar{P}_{\gn}\!)}}$ has dimension $\dim\eW^{\mathsmaller{(\!\bar{P}_{\gn}\!)}}=\dim\widehat{\eW^{\mathsmaller{(\!\bar{P}_{\gn}\!)}}}=\nu+q$.

Let now $P(t) = P_{\gt}(t) \oplus P_{\gn}(t)$ be an ancient solution to the $\bar{P}_{\gn}$-projected Ricci flow lying on $\eW^{\mathsmaller{(\!\bar{P}_{\gn}\!)}}$. It remains to prove that the corresponding solution to the Ricci flow is still ancient. Notice that by \eqref{limitscal} it holds that
$$
\scal_{\mathsmaller M}(P(t)) \to \l \dim(N) \quad \text{ as } t \to -\infty \,\, . 
$$
Thus for large times $\scal_{\mathsmaller M}(P(t))>0$, and hence the same is true for the corresponding solution to the Ricci flow. However, a solution to the Ricci flow whose scalar curvature stays positive is necessarily ancient by \cite[Thm.\ 1.1]{La}. Furthermore, as in the proof of Proposition 4.2 in \cite{WZ}, $P(t)$ has bounded curvature and is hence collapsed as the injectivity radius tends to zero.
\end{proof}

Finally, Theorem \ref{main-A} is a direct consequence of Theorem \ref{thm:existence} and Proposition \ref{propconvGH}.

\medskip
\section{Proof of Corollary \ref{main-B}} \label{sect:examples}
\setcounter{equation} 0

In this section, we produce explicit examples of collapsed homogeneous ancient solutions. As a byproduct, we prove Corollary \ref{main-B}. For a detailed study of Einstein equations on generalized flag manifolds, we refer {\it e.g.\ }to \cite{Arv}. In the following examples, the group $\fG$ will always be semisimple and so we choose its negative Cartan-Killing form as background metric.

\subsection{A K\"ahler-Einstein metric on $\fSU(3)/\fT^2$} \hfill \par

Let $\fG=\fSU(3)$, $\fT^2 = \{\operatorname{diag}(e^{it_1},e^{it_2},e^{-i(t_1+t_2)})\}$ its maximal torus and consider the real root spaces decomposition
$$
\su(3) = \gt^2 + \gn_1 +\gn_2 +\gn_3 \,\, .
$$
Then, any $\fG$-invariant Riemannian metric $P_{\gn}$ on the flag manifold $N = \fSU(3)/\fT^2$ takes the form
$$
P_{\gn} = \l_1 \Id_{\gn_1} \oplus \l_2 \Id_{\gn_2} \oplus \l_3 \Id_{\gn_3} 
$$
and its normalized scalar curvature is given by {(see {\it e.g.\ }\cite[Prop.\ 4]{Arv})}
$$
\wt{\scal}_{\mathsmaller N}(P_{\gn}) = (\l_1\l_2\l_3)^{\frac13} \Big( \tfrac1{\l_1} +\tfrac1{\l_2} +\tfrac1{\l_3} -\tfrac16 \big( \tfrac{\l_1}{\l_2\l_3} +\tfrac{\l_2}{\l_1\l_3} +\tfrac{\l_3}{\l_1\l_2} \big) \Big) \,\, .
$$
Take the unit volume K\"ahler-Einstein metric $P_{\gn}^{\rm \mathsmaller{KE}}$ corresponding to the values
$$
(\l_1,\l_2,\l_3) = \tfrac13(\tfrac{27}{2})^{\frac13}(1,1,2)
$$
and one can compute that
$$
\operatorname{spectrum}\!\Big(\Hess\!\big(\wt{\scal}_{\mathsmaller N}\big)\big|_{P_{\gn}^{\rm \mathsmaller{KE}}}\Big) = \big\{{-}\tfrac13, 0, \tfrac43\big\} \,\, .
$$
Here, the zero eigenvalue corresponds to scaling the metric by a constant and so $P_{\gn}^{\rm \mathsmaller{KE}}$ has coindex $q=1$.
Now consider the homogeneous fibration
\begin{equation} \label{SU(3)fibration1}
\fT^2 \to \fSU(3) \to \fSU(3)/\fT^2 \,\, .
\end{equation}
By Theorem \ref{thm:existence}, there is a $3$-parameter family of ancient solutions to the Ricci flow on $\fSU(3)$ which, under the rescaling introduced in Section \ref{sect:projRF}, collapse the fibers of \eqref{SU(3)fibration1} and converge to $0 \oplus P_{\gn}^{\rm \mathsmaller{KE}}$ as $t\to-\infty$. Similarly if $\fS^1_{p,q}=\{\operatorname(e^{i pt},e^{i qt},e^{-i(p+q)t})\}$, we get the homogeneous fibration
\begin{equation} \label{SU(3)fibration2}
\fS^1=\fT^2/\fS^1_{p,q} \to \fSU(3)/\fS^1_{p,q} \to \fSU(3)/\fT^2 \,\, ,
\end{equation}
where $\fSU(3)/\fS^1_{p,q}$ is an Aloff-Wallach space. By Theorem \ref{thm:existence}, there is a 1-parameter family of ancient solutions on $\fSU(3)/\fS^1_{p,q}$ converging to $0 \oplus P_{\gn}^{\rm \mathsmaller{KE}}$ as above.

\begin{rem}
In \cite{LW}, Lu and Wang produce a two-parameter family of ancient solutions on $\fSU(3)$ and a single ancient solution on $\fSU(3)/\fS^1_{p,q}$ both collapsing to $P_{\gn}^{\rm \mathsmaller{KE}}$ as $t\to-\infty$. Our families are slightly larger, which can be explained by the fact that the metric restricted to the base is allowed to vary.
\end{rem}

\subsection{A K\"ahler-Einstein metric on $\fSU(4)/\fT^3$} \hfill \par

Let $\fG=\fSU(4)$, $\fT^3 = \{\operatorname{diag}(e^{it_1},e^{it_2},e^{it_3},e^{-i(t_1+t_2+t_3)})\}$ its maximal torus and consider the real root spaces decomposition
$$
\su(3) = \gt^2 +\gn_1 +\gn_2 +\gn_3 +\gn_4 +\gn_5 +\gn_6 \,\, .
$$
Then, any $\fG$-invariant Riemannian metric $P_{\gn}$ on the flag manifold $N = \fSU(4)/\fT^3$ takes the form
$$
P_{\gn} = \l_1 \Id_{\gn_1} \oplus {\dots} \oplus \l_6 \Id_{\gn_3} 
$$
and its normalized scalar curvature is given by {(see {\it e.g.\ }\cite[Prop.\ 4]{Arv})}
\begin{multline*}
\wt{\scal}_{\mathsmaller N}(P_{\gn}) = (\l_1\l_2\l_3\l_4\l_5\l_6)^{\frac16}\Big(
\tfrac1{\l_1} +\tfrac1{\l_2} +\tfrac1{\l_3} +\tfrac1{\l_4} +\tfrac1{\l_5} +\tfrac1{\l_6}
-\tfrac18\big(\tfrac{\l_1}{\l_2\l_4} +\tfrac{\l_1}{\l_3\l_5} +\tfrac{\l_2}{\l_1\l_4} +\tfrac{\l_2}{\l_3\l_6} \\
+\tfrac{\l_3}{\l_1\l_5} +\tfrac{\l_3}{\l_2\l_6} +\tfrac{\l_4}{\l_1\l_2} +\tfrac{\l_4}{\l_5\l_6} +\tfrac{\l_5}{\l_1\l_3} +\tfrac{\l_5}{\l_4\l_6} +\tfrac{\l_6}{\l_2\l_3} +\tfrac{\l_6}{\l_4\l_5}\big)\Big) \,\, .
\end{multline*}
Take the unit volume K\"ahler-Einstein metric $P_{\gn}^{\rm \mathsmaller{KE}}$ corresponding to the values
$$
(\l_1,\l_2,\l_3,\l_4,\l_5,\l_6) = \tfrac14(\tfrac{1024}{3})^{\frac16}(3,2,1,1,2,1).
$$
One can compute that the matrix $\Hess\!\big(\wt{\scal}_{\mathsmaller N}\big)\big|_{P_{\gn}^{\rm \mathsmaller{KE}}}$ has two distinct positive eigenvalues, three negative eigenvalues and one zero eigenvalue, corresponding to the scaling direction. Therefore, $P_{\gn}^{\rm \mathsmaller{KE}}$ has coindex $q=2$. By Theorem \ref{thm:existence}, on $\fSU(4)$ there is a $7$-parameter family of ancient solutions to the Ricci flow collapsing, under rescaling, to $0 \oplus P_{\gn}^{\rm \mathsmaller{KE}}$ as $t\to-\infty$. Similarly on $\fSU(4)/\fS^1$ there is a $4$-parameter family of ancient solutions, and on $\fSU(4)/\fT^2$ there is a $2$-parameter family of ancient solutions.

Notice that, as in the previous example, the construction of Lu and Wang again provides ancient solutions on $\fSU(4)$ but their family is two dimensions smaller, due to the fact that in their construction the metric on the base remains fixed.

\subsection{A K\"ahler-Einstein metric on $\fG_2/\fT^2$} \hfill \par

Let $\fG=\fG_2$, $\fT^2$ a maximal torus inside $\fG_2$ and consider the real root spaces decomposition
$$
\su(3) = \gt^2 +\gn_1 +\gn_2 +\gn_3 +\gn_4 +\gn_5 +\gn_6 \,\, .
$$
Then, any $\fG$-invariant Riemannian metric $P_{\gn}$ on the flag manifold $N = \fG_2/\fT^2$ takes the form
$$
P_{\gn} = \l_1 \Id_{\gn_1} \oplus {\dots} \oplus \l_6 \Id_{\gn_6} 
$$
and its normalized scalar curvature is given by
\begin{multline*}
\wt{\scal}_{\mathsmaller N}(P_{\gn}) = (\l_1\l_2\l_3\l_4\l_5\l_6)^{\frac16} \Big(
\tfrac1{\l_1} +\tfrac1{\l_2} +\tfrac1{\l_3} +\tfrac1{\l_4} +\tfrac1{\l_5} +\tfrac1{\l_6} -\tfrac16\big(\tfrac{\l_1}{\l_3\l_4} +\tfrac{\l_3}{\l_1\l_4} +\tfrac{\l_4}{\l_1\l_3}\big) \\
-\tfrac18\big(\tfrac{\l_1}{\l_2\l_3} +\tfrac{\l_1}{\l_4\l_5} +\tfrac{\l_2}{\l_1\l_3} +\tfrac{\l_2}{\l_5\l_6} +\tfrac{\l_3}{\l_1\l_2} +\tfrac{\l_3}{\l_4\l_6} +\tfrac{\l_4}{\l_1\l_5} +\tfrac{\l_4}{\l_3\l_6} +\tfrac{\l_5}{\l_1\l_4} +\tfrac{\l_5}{\l_2\l_6} +\tfrac{\l_6}{\l_2\l_5} +\tfrac{\l_6}{\l_3\l_4} \big) \Big).
\end{multline*}
For more information about homogeneous Einstein metrics on $N = \fG_2/\fT^2$, see \cite{ACS}. Take the unit volume K\"ahler-Einstein metric $P_{\gn}^{\rm \mathsmaller{KE}}$ corresponding to the values
$$
(\l_1,\l_2,\l_3,\l_4,\l_5,\l_6) = \tfrac1{12}(\tfrac{4608}{5})^{\frac16}(1,3,4,5,6,9)
$$
and one can compute that the matrix $\Hess\!\big(\wt{\scal}_{\mathsmaller N}\big)\big|_{P_{\gn}^{\rm \mathsmaller{KE}}}$ has one positive eigenvalue, four negative eigenvalues and one zero eigenvalue, corresponding to the scaling direction. Therefore, $P_{\gn}^{\rm \mathsmaller{KE}}$ has coindex $q=1$. By Theorem \ref{thm:existence}, on $\fG_2$ there is a $3$-parameter family of ancient solutions to the Ricci flow collapsing, under rescaling, to $P_{\gn}^{\rm \mathsmaller{KE}}$ as $t\to-\infty$. Similarly on $\fG_2/\fS^1$ there is a $1$-parameter family of ancient solutions.

\subsection{The normal Einstein metric on $\fSU(n)/\fT^{n-1}$} \hfill \par

Let $\fG=\fSU(n)$, with $n \geq 3$, and $\fT^{n-1} \subset \fSU(n)$ the diagonally embedded maximal torus. Then for any $1\leq k\leq n-1$ and any subtorus $\fT^{n-1-k}\subset \fT^{n-1}$ we have a homogeneous fibration
$$
\fT^k \to \fSU(n)/\fT^{n-1-k} \to  \fSU(n)/\fT^{n-1} \,\, ,
$$
where $\fT^k$ is a complement of $\fT^{n-1-k}$ in $\fT^{n-1}$. By \cite{Lau}, the normal metric on $\fSU(n)/\fT^{n-1}$ induced by the biinvariant metric on $\fSU(n)$ is Einstein with coindex $q= n-1$. Hence by Theorem \ref{thm:existence}, there exists a $\big(\frac{k(k+1)}2+n-2\big)$-parameter family of ancient solutions on $\fSU(n)/\fT^{n-1-k}$ which collapse, under rescaling, to the normal metric on the base as $t \to -\infty$.

\subsection{The normal Einstein metric on $\fSO(4)/\fT^2= S^2\times S^2$}  \hfill \par

Let $\fG=\fSO(4)$, $\fT^2\subset\fSO(4)$ be a maximal torus and consider the $\Ad(\fT^2)$-irreducible decomposition
$$
\so(4)=\gt^2+\gn_1+\gn_2.
$$
Then, any $\fG$-invariant Riemannian metric $P_\gn$ on $N = \fSO(4)/\fT^2$ takes the form 
$$
P_{\gn} = \l_1 \Id_{\gn_1} \oplus \l_2 \Id_{\gn_2} 
$$
and its normalized scalar curvature is given by 
$$
\wt{\scal}_{\mathsmaller N}(P_{\gn}) = (\l_1\l_2)^{\frac12}\Big(\tfrac1{\l_1}+\tfrac1{\l_2}\Big).
$$
The normal metric $P_{\gn}^{\rm \mathsmaller{E}}$ induced by the biinvariant metric on $\fSO(4)$ is Einstein and given by
$$
(\l_1,\l_2)=(1,1) \,\, .
$$
One can compute that the matrix $\Hess\!\big(\wt{\scal}_{\mathsmaller N}\big)\big|_{P_{\gn}^{\rm \mathsmaller{E}}}$ has one positive eigenvalue and one zero eigenvalue, corresponding to the scaling direction. Therefore, $P_{\gn}^{\rm \mathsmaller{E}}$ has coindex $q=1$. Hence by Theorem \ref{thm:existence} on $\fSO(4)$ there is a $3$-parameter family of ancient solutions which collapse, under rescaling, to $P_{\gn}^{\rm \mathsmaller{E}}$ as $t\to-\infty$. Similarly, if $\fS^1_{p,q}\subset\fT^2$ is a diagonally embedded circle with rational slope $\tfrac{p}{q}$, then on $\fSO(4)/\fS^1_{p,q}\simeq S^3\times S^2$ there is a $1$-parameter family of ancient solutions which collapse, under rescaling, to $P_{\gn}^{\rm \mathsmaller{E}}$ as $t\to-\infty$.

\end{document}